\documentclass[11pt]{article}
\usepackage[utf8]{inputenc}
\usepackage[T1]{fontenc}
\usepackage{lmodern}
\usepackage{geometry}
\author{David Forsman\\
Université catholique de Louvain\\
\texttt{david.forsman@uclouvain.be}}
\title{A Categorical Approach to Finiteness Conditions}
\date{12 September 2025}

\usepackage{amsthm, amssymb, mathtools}
\usepackage{enumitem}
\usepackage{graphicx}
\usepackage{multicol} 

\usepackage{tikz-cd}
\usepackage{caption, subcaption}
\usepackage{microtype}
\usepackage[numbers]{natbib}
 \usepackage{hyperref}
\hypersetup{
    breaklinks=true,   
    colorlinks=true,   
    linkcolor=blue,    
    citecolor=blue,    
 }

 \usepackage{cleveref}

\theoremstyle{plain}
\newtheorem{theorem}{Theorem}[section]
\newtheorem{lemma}[theorem]{Lemma}
\newtheorem{corollary}[theorem]{Corollary}

\theoremstyle{definition}
\newtheorem{definition}[theorem]{Definition}
\newtheorem{example}[theorem]{Example}
\newtheorem{remark}[theorem]{Remark}

\numberwithin{equation}{section}

\newcommand{\N}{\mathbb{N}}

\begin{document}
\maketitle
      
\begin{abstract}
We introduce a general categorical framework for finiteness conditions that unifies classical notions such as Noetherianness, Artinianness, and various forms of topological compactness. This is achieved through the concept of \textbf{$\tau$-compactness}, defined relative to a \textbf{coverage} $\tau$. A coverage on a category $C$ is a specified class of covering diagrams, which are functors $F\colon I \to C/c$ of a specified variance, where the indexing category $I$ is equipped with a set of 'designated small objects'. An object $c$ is $\tau$-compact if every such covering diagram over it stabilizes at some designated small object. As we permit functors of mixed variance, our framework simultaneously models ascending chain conditions (such as Noetherianness) and descending chain conditions (such as Artinianness, topological compactness via closed sets).

The role of protomodularity of the ambient category emerges as a crucial property for proving strong closure results. Under suitable compatibility assumptions on the coverage, we show that in a protomodular category, the class of $\tau$-compact objects is closed under quotients and extensions. In a pointed context, this implies closure under finite products, generalizing the classical theorem that a finitely generated module over a Noetherian ring is itself Noetherian. We also leverage protomodularity to establish a categorical Hopfian property for Noetherian objects.

Our main application shows that in any regular protomodular category with an initial object, the classes of Noetherian and Artinian objects are closed under subobjects, regular quotients, and extensions. As a consequence, in any abelian category, these classes of objects form exact subcategories.
\end{abstract}
    
\section*{Introduction}
\addcontentsline{toc}{section}{Introduction}

This paper develops a categorical notion of compactness, $\tau$-compactness, defined relative to a \emph{coverage} $\tau$ on a category $C$. The framework generalizes classical finiteness conditions, including Noetherian and Artinian conditions in algebra, as well as several compactness notions in topology, such as compactness, Lindelöfness, and countable compactness.

A coverage assigns to each object $c$ a class of functors $F\colon I\to C/c$ of mixed variance indexed by a small category $I$ endowed with designated small objects. An object $c$ is $\tau$-compact if every such covering functor over $c$ stabilizes at some designated small object. Allowing functors to be of mixed variance recovers both ascending and descending chain conditions within the same formalism. Topological compactness of a topological space $X$ can be recovered via stabilization of covariant functors $\mathcal{P}(I)\to \textbf{Top}/X$ arising from $I$-indexed open covers of $X$, where $\mathcal{P}(I)$ is the power set of $I$ with its finite subsets designated as small. We may also consider contravariant functors $\mathcal{P}(I)\to \textbf{Top}/X$ arising from $I$-indexed families, whose intersection is empty, of closed subsets of $X$. The stabilization with respect to finite subsets also recovers the classical notion of compactness.

Our main results rely on protomodularity. Under natural compatibility assumptions, in a protomodular category, the $\tau$-compact objects are closed under regular quotients and extensions (Theorem~\ref{thm:Closure}); in pointed contexts, they are closed under finite products (Corollary~\ref{cor:closureproducts}). We also establish a categorical Hopfian property for Noetherian objects (Theorem~\ref{thm:Hopfian}). As an application, in any regular protomodular category with an initial object, the Noetherian objects are Hopfian and closed under subobjects, regular quotients, and extensions; consequently, in any abelian category, these classes form exact subcategories (Corollary~\ref{cor:noeth-regular-protomod-initial}).

Section~\ref{sec:Systems} recalls factorization systems, regularity, protomodularity, and develops functors of mixed variance. Section~\ref{sec:DirectedCoverageCompactnessNoetherianness} introduces coverages and $\tau$-compactness and proves the closure and Hopfian results together with examples. An appendix gathers further examples of classes of morphisms satisfying the relative protomodularity condition.

\section{Factorizations, Regularity, Protomodularity and Variance} \label{sec:Systems}

We begin by recalling properties of classes of morphisms needed for the theory of coverages $\tau$ and $\tau$-compactness, and then the standard notions surrounding orthogonal factorization systems. Our goal is to build toward \emph{protomodular factorization systems}, concentrating regularity and protomodularity into a property of a factorization system. Regularity is captured by a stable orthogonal factorization system $(E,M)$, while protomodularity appears via a variant of the short five lemma for diagrams involving $E$ and $M$. We give examples arising from arbitrary algebraic theories where suitable pairs $(E,M)$ satisfy the protomodularity condition.

\subsection{Factorizations and Regularity}
We define properties for classes of morphisms, then recall orthogonal factorization systems, and specialize to regular categories, which inherently carry a stable orthogonal factorization system.

\begin{definition}[Properties of Morphism Classes] \label{def:MorphismClassProperties}
Let $C$ be a category with pullbacks and let $A$ be a class of morphisms in $C$.
\begin{enumerate}
    \item $A$ is a \textit{system} of $C$ if $A$ contains all isomorphisms of $C$ and is closed under composition.
    \item $A$ is \textit{stable} if all pullbacks of morphisms in $A$ are in $A$ along any morphisms in $C$.
    \item A morphism $f$ in $C$ is said to be \textit{stably in $A$} if every pullback of $f$ is in $A$. Similarly, $f$ \textit{stably satisfies property $P$} if every pullback of $f$ satisfies $P$.
    \item We call a family of morphisms $(f_i\colon x_i\to x)_{i\in I}$ in $C$ an $A$-\textit{extremal epimorphic family}, if $f_i = mg_i$ implies that $m$ is an isomorphism whenever $m\in A$ and $g_i$ is an arbitrary morphism for $i\in I$. If the condition holds for a family of a single morphism, then we call the morphism an $A$-\textit{extremal epimorphism}. We define $A$-extremal monicness dually.
    \item $A$ is \textit{left-cancelable} if $g \circ f \in A$ and $g \in A$ implies that $f \in A$ for any composable pair $f,g$ in $C$. \textit{Right-cancelability} is defined dually.
    \item A pair $(E,M)$ of systems of $C$ with pullbacks is called \textit{stable} if both $E$ and $M$ are stable.
\end{enumerate}
\end{definition}

\begin{definition}[Lifting Problems and Factorization Systems] \label{def:LiftingFactorization}
Let $C$ be a category, and let $E$ and $M$ be classes of morphisms in $C$.
\begin{enumerate}
    \item A \textit{lifting problem} for a pair $(e\colon A \to B, m\colon X \to Y)$ is a commutative square:
    $$
    \begin{tikzcd}
    A \arrow[r, "u"] \arrow[d, "e"'] & X \arrow[d, "m"] \\
    B \arrow[r, "v"'] & Y
    \end{tikzcd}
    $$
    A morphism $h\colon B \to X$ such that $h \circ e = u$ and $m \circ h = v$ is a \textit{solution} (or \textit{lift}). We write $e \perp m$ if every lifting problem for $(e,m)$ has a unique solution. We extend this notation to classes: $E \perp M$ if $e \perp m$ for all $e \in E, m \in M$.

    \item The pair $(E,M)$ is a \textit{weak factorization system} if:
        \begin{enumerate}
            \item Every morphism $f$ in $C$ factors as $f = m \circ e$ with $e \in E$ and $m \in M$.
            \item $E =\ ^{\bot}\! M\coloneq\{e \in \mathrm{Mor}(C) \mid e \perp M \}$ and $M = E^\bot \coloneqq \{m \in \mathrm{Mor}(C) \mid E \perp m \}$.
        \end{enumerate}
\end{enumerate}
\end{definition}

\begin{definition}
A category $C$ is called \textit{regular}, if it is finitely complete and each morphism $f$ in $C$ factors $f = me$ where $m$ is monic and $e$ is a stably extremal epimorphism.
\end{definition}

\begin{theorem}
    Let $C$ be a regular category with the classes $E$ and $M$ of extremal epimorphisms and monomorphisms, respectively. Then $(E,M)$ is a stable orthogonal factorization system and $E$ is the class of regular epimorphisms. Espeically, each kernel pair has a coequalizer.
\end{theorem}
\begin{proof}
    It is clear that $(E,M)$ is a stable orthogonal factorization system, and since $M$ is the class of monomorphisms, it follows by \cite{Kelly1991} that $E$ is the class of regular epimorphisms.
\end{proof}

\subsection{Protomodular Systems}
Different generalizations of protomodular categories have been of interest by controlling which morphism satisfy the short-five-lemma \cite{RelativeHomologicalCategories, MonoidsAndPointedSProtomodular}. We propose a new idea to this direction: For classes $E$ and $M$ of morphisms of a category $C$ with pullbacks, we define what we mean by $E$ satisfying the $M$-protomodularity condition. When $E$ is the class of retractions and $M$ the class of all morphisms (or equivalently the class of monomorphisms), the $M$-protomodularity of $E$ recovers the notion that $C$ is a protomodular category in the sense of Bourn \cite{BorceuxBourn2004}.
\begin{definition} \label{def:SystemProperties}
Let $C$ be a category with pullbacks. Let $E$ and $M$ be classes of morphisms in $C$. We say that $E$ satisfies the $M$-\textit{protomodularity condition} or the pair $(E,M)$ satisfies the protomodularity condition if for any commutative diagram
        $$
\begin{tikzcd}
a' \arrow[d, "\alpha"'] \arrow[r, "m'"] \arrow[rd, "\lrcorner", phantom, very near start] & b' \arrow[d, "\beta" description] \arrow[r, "e'"] & c' \arrow[d, "\gamma"] \\
a \arrow[r, "m" description] \arrow[rd] \arrow[rrd, very near start, phantom, "\lrcorner"]                                                  & b \arrow[r, "e"']                                 & c                      \\
                                                                                          & I \arrow[ru, "\theta"']                           & {}                    
\end{tikzcd}
        $$
        in $C$, where the left square and lower triangle form pullbacks, $e,e'\in E, \beta\in M$ and $\alpha$ and $\gamma$ are isomorphisms, then $\beta$ is an isomorphism. We say that $E$ satisfies the \textit{protomodularity condition}, if it satisfies the $M'$-protomodularity condition, where $M'$ is the class of all morphisms in $C$. The category $C$ is said to be a \textit{protomodular} category, if the class of retractions satisfies the protomodularity condition.
\end{definition}
\begin{remark}\label{rmk:equivalentformofprotomodularity}
    Notice that the $M$-protomodularity condition for a system $E$ of $C$, where $C$ has pullbacks, can be equivalently formulated as follows:
    Let the following be a commutative rectangle in $C$:
    $$
    \begin{tikzcd}
a' \arrow[r, "\cong"] \arrow[d] \arrow[rd, "\lrcorner", phantom, very near start] & a \arrow[r] \arrow[d] \arrow[rd, "\lrcorner", phantom, very near start] & I \arrow[d, "\theta"] \\
b' \arrow[r, "\beta"']                     & b \arrow[r, "e"']                & c                    
\end{tikzcd}
    $$
    If $e,e\beta\in E$, $\beta\in M$, then $\beta$ is an isomorphism. From this formulation, it is easy to see that if $C$ has an initial object $0$, the object $I$ can be replaced by the initial object $0$, and the definition of $M$-protomodularity stays equivalent. If $\theta$ has a factorization $\theta = me$ where $m$ is monic and $e$ is a stably extremal epimorphism, then $\theta$ may be replaced with $m$, and the condition stays equivalent.
\end{remark}

There is also a compact way of defining protomodularity of $C$: Every retraction $r\colon x\to y$ jointly generates $x$ via $(s,k)$, where $s\colon y\to x$ is any splitting of $r$ and $k$ any pullback along $r$.

\begin{example}
    Let $C$ be a category with pullbacks. The classes of retractions and stably extremal epimorphisms of $C$ form stable right-cancelable systems of $C$. It is known that if $C$ is a regular protomodular category, then the class $E$ of regular epimorphisms forms a stable, right-cancelable and protomodular system of $C$. It is perhaps less known that if $C$ is a protomodular category, then the class $E$ of stably extremal epimorphisms forms a stable, right-cancelable and protomodular system of $C$. The proof follows similarly to the regular case. One may consider the results Lemma 1.15 and Theorem 1.16 in \cite{inbook} and Theorem 4.1.4(2) in \cite{BorceuxBourn2004} and generalize them slightly to attain the protomodularity of $E$. 
\end{example}
\begin{lemma}\label{lem:reflecting-protomodularity}
    Let $F_i\colon C\to D_i,i\in I,$ be a jointly conservative family of pullback preserving functors. Assume that $C$ and $D_i$ are categories with pullbacks for $i\in I$. Let $(E_i,M_i)$ be a pair of classes of morphisms of $D_i$ satisfying the protomodularity condition for $i\in I$. Then $(E',M')$ satisfies the protomodularity condition, where $E' = \bigcap_{i\in I} F_i^{-1}(E_i)$ and $M' = \bigcap_{i\in I} F_i^{-1}(M_i)$. In particular, if $D_i$ is protomodular for each $i\in I$, so is $C$.
\end{lemma}
\begin{proof}
Consider the following commutative diagram in $C$
        $$
\begin{tikzcd}
a' \arrow[d, "\alpha"'] \arrow[r, "m'"] \arrow[rd, "\lrcorner", phantom, very near start] & b' \arrow[d, "\beta" description] \arrow[r, "e'"] & c' \arrow[d, "\gamma"] \\
a \arrow[r, "m" description] \arrow[rd] \arrow[rrd, very near start, phantom, "\lrcorner"]                                                  & b \arrow[r, "e"']                                 & c                      \\
                                                                                          & I \arrow[ru, "\theta"']                           & {}                    
\end{tikzcd}
        $$
        where the left square and lower triangle form pullbacks. Assume that $e,e'\in E'$, $\beta\in M'$ and $\alpha$ and $\gamma$ are isomorphisms. By the assumptions, we have that $F_i(\beta)$ is an isomorphism for each $i\in I$ and by joint conservativity, $\beta$ is an isomorphism.
\end{proof}

\begin{definition}
    Let $T$ be an algebraic theory with a signature $\sigma$. We call the theory $T$
    \begin{enumerate}
        \item a \textit{pointed} theory, if there is a constant symbol $e$ in $\sigma$, where $T$ proves $f(e,\ldots, e) \approx e$ for each function symbol $f$ in $\sigma$.
        \item a \textit{Malcev theory}, if there is a \textit{Malcev term} $p(x,y,z)$, where $T$ proves
        $$
        p(x,y,y)\approx x \quad \text{and}\quad p(x,x,y) \approx y.
        $$
        \item a \textit{protomodular theory}, if there exists $n\in\N$ and terms $\theta(x_0,\ldots, x_n), \theta_i(x,y)$ and a constant term $e_i$ for $0<i\leq n$ such that $T$ proves
        $$
        \theta_1(x,x)\approx e_1,\ldots, \theta_n(x,x)\approx e_n\quad \text{and}\quad \theta(y,\theta_1(x,y),\ldots, \theta_n(x,y)) \approx x.
        $$
        \item a \textit{semi-abelian} theory, if it is both pointed and protomodular.
    \end{enumerate} 
\end{definition}
The category $\textbf{Set}^T$ is pointed or protomodular if and only if $T$ is a pointed or a protomodular algebraic theory, respectively \cite{BorceuxBourn2004}. Each protomodular theory $T$ is also a Malcev theory, since it has a Malcev term $p(x,y,z) = \theta(x,\theta_1(y,z),\ldots, \theta_n(y,z))$. 

If the theory $T$ is protomodular and $C$ is a category with finite limits, then the category $C^T$ of $T$-algebras in $C$ forms a protomodular category. This follows from the fact that the representable functors $C \to \textbf{Set}$ induce a jointly conservative family of continuous functors $C^T\to \textbf{Set}^T$. By Lemma \ref{lem:reflecting-protomodularity}, $C^T$ is a protomodular category. Some examples of protomodular categories include the category \textbf{Grp} of groups, the opposite category of any elementary topos, and logical categories of Heyting algebras or hoops.
symmetricsymmetric
\subsection{Functors of Mixed Variance}
We give a quick introduction to functors of mixed variance. We want to generalize both covariant and contravariant functors $F\colon C\to D$, and this is done by introducing a notion of variance $(A,B)$, which is a two-sided strict factorization system, on the category $C$. For a variance $(A,B)$, the classes $A$ and $B$ represent those morphisms of $C$ that we consider as covariant or contravariant morphisms, respectively. We then define what is meant by a functor $F$ of variance $(A,B)$. The intuition is that a functor of variance $(A,B)$ behaves covariantly on $A$, contravariantly on $B$, and respects the two factorizations each morphism $f$ in $C$ has. 
\begin{definition}
    Let $C$ be a category with classes of morphisms $A$ and $B$. We call the pair $(A,B)$ a strict factorization system on $C$ if:
    \begin{enumerate}
        \item The classes $A$ and $B$ form wide subcategories of $C$; they are closed under composition and contain all identity morphisms of $C$.
        \item Each morphism $f$ in $C$ factors uniquely $f = ba$ where $a\in A$ and $b\in B$. 
    \end{enumerate}
    We call the pair $(A, B)$ a variance on $C$ if both $(A,B)$ and $(B,A)$ are strict factorization systems on $C$. We then denote the factorization of a morphism $f\colon x\to y$ via the commutative diagram
    $$
    \begin{tikzcd}
x \arrow[r, "f^+"] \arrow[d, "f_-"'] \arrow[rd, "f" description] & {f_t} \arrow[d, "f^{-}"] \\
f_s \arrow[r, "f_+"']                                            & y                    
\end{tikzcd}
    $$
    where $f_+,f^+\in A$ and $f_-,f^-\in B$. We call $(C,A,B)$ a category with a variance $(A,B)$, $A$ the class of covariant morphisms, and $B$ the class of contravariant morphisms.
    \end{definition}
    \begin{example} Here are some examples of variances:
    \begin{enumerate}
        \item Each category $C$ has variances $(A,B)$ and $(B,A)$, where $A$ is the class of all morphisms and $B$ is the class of all identity morphisms of $C$. We call $(A,B)$ the \textit{covariant variance} of $C$ and $(B,A)$ the contravariant variance of $C$.
        \item If $C$ and $D$ have variances $(A,B)$ and $(A',B')$, respectively, then $C\times D$ has the variance $(A\times A',B\times B')$. Similarly, $(A+A',B+B')$ is a variance on the coproducts $C+D$.
        \item A strict factorization system $(A,B)$ on a groupoid is always a variance and the subcategories $A,B$ are subgroupoids. Thus Zappa-Szép products, which are strict factorization systems generalizing semi-direct products, on a group $G$ correspond bijectively to variances on the group $G$.
    \end{enumerate}
    \end{example}

The following definition generalizes multi-functors which are covariant in some components and contravariant in others. 
\begin{definition}
    Let $C$ be a category with a variance $(A,B)$. Let $D$ be a category. We call $F\colon C\to D$ a functor of variance $(A,B)$, if $F$ determines functions from the objects and morphisms of $C$ to the objects and morphisms of $D$ such that $F$ maps identities to identities and respects the composition in the following sense: If $x\xrightarrow{f} y\xrightarrow{g} z$ is a composable pair in $C$, then $f\xmapsto{F} F(f)\colon F(f_s)\to F(f_t)$ and the diagram commutes:
    $$
    \begin{tikzcd}
                                                                                                    & F(f_s) \arrow[r, "F(f)"] & F(f_t) \arrow[rd, "F((g^+f^-)^+)"]  &           \\
F((gf)_s) \arrow[ru, "F((g_-f_+)_-)"] \arrow[rd, "F((g_-f_+)_+)"'] \arrow[rrr, "F(gf)" description] &                          &                                     & F((gf)_t) \\
                                                                                                    & F(g_s) \arrow[r, "F(g)"] & F(g_t) \arrow[ru, "F((g^+f^-)^-)"'] &          
\end{tikzcd}
    $$
    A \textit{natural transformation} $\eta\colon F\Rightarrow G\colon C\to D$ between functors $F$ and $G$ of variance $(A,B)$ consists of a family of morphisms $(\eta_c\colon Fc\to Gc)_{c\in\mathrm{Obj}(C)}$ forming a commutative square
    $$
    \begin{tikzcd}
F(f_s) \arrow[r, "\eta_{f_s}"] \arrow[d, "F(f)"'] & G(f_s) \arrow[d, "G(f)"] \\
F(f_t) \arrow[r, "\eta_{f_t}"']                   & G(f_t)                  
\end{tikzcd}
    $$
    for every morphism $f\colon c\to c'$ in $C$.

\end{definition}
\begin{theorem}
    Let $C$ and $D$ be categories with $C$ having a variance $(A,B)$. Then the association $F\mapsto (F\mid A,F\mid B)$ defines a bijective correspondence between functors $F\colon C\to D$ of variance $(A,B)$ and pairs $(G\colon A\to D,H\colon B\to D)$, where $G$ and $H$ are covariant and contravariant functors, respectively, satisfying the following:
    \begin{itemize}
        \item $G(c) = H(c)$ for each object $c$ in $C$.
        \item For any morphism $f\colon c\to c'$ in $C$, we have a commutative diagram
        $$
        \begin{tikzcd}
Gc \arrow[r, "G(f^+)"]                        & Gf_t                     \\
Gf_s \arrow[r, "G(f_+)"'] \arrow[u, "H(f_-)"] & Gc' \arrow[u, "H(f^-)"']
\end{tikzcd}
        $$
    \end{itemize}
    \end{theorem}
\begin{proof}
    A proof is found in \cite{forsman2023functorsvarianceheuristicnaturality}.
\end{proof} 

\begin{lemma}\label{lem:pullbackinduced}
    Let $C$ be a category with pullbacks and a morphism $f\colon x\to y$. Then the following holds:
    \begin{enumerate}
        \item Let $G\colon I\to C/y$ be a functor of variance $(A,B)$. Then the pullbacks along $f$ induce a functor $f^*(G)\colon I\to C/x$ of variance $(A,B)$. Moreover, the pullbacks of $f$ determine a natural transformation $f_*f^*(G)\Rightarrow G$, where $f_*\colon C/x\to C/y$ is the induced functor between the slice categories.
        \item Let $(E,M)$ be an orthogonal factorization system on $C$. Let $F\colon I\to C/x$ be a functor of variance $(A,B)$. Then the $(E,M)$-factorizations induce a functor $G\colon I\to C/y$ with a natural transformation $\eta\colon f_*F\Rightarrow G\colon I\to C/y$, where $\eta_i\in E$ and $G(i)\in M$ for each object $i$ in $I$.
    \end{enumerate}
\end{lemma}
\begin{proof}
    \hfill
    \begin{enumerate}
        \item We define the functor $F\colon I\to D_x$ of variance $(A,B)$ by setting $F(c)\colon x_c\to x$ to be the pullback of $G(c)$ along $f$ and $G(k)\colon G(k_s)\to G(k_t)$ to be the morphism induced by the pullbacks of $F(k_s)$ and $F(k_t)$ along $f$: The diagram
        $$
        \begin{tikzcd}
x_{k_s} \arrow[d, "F(k)" description] \arrow[r, "f_{k_s}"] \arrow[rd, very near start, phantom, "\lrcorner"]   & y_{k_s} \arrow[d, "G(k)"]   \\
x_{k_t} \arrow[d, "F(k_t)" description] \arrow[rd, very near start, phantom, "\lrcorner"] \arrow[r, "f_{k_t}"] & y_{k_t} \arrow[d, "G(k_t)"] \\
x \arrow[r, "f"']                                                       & y                          
\end{tikzcd}
        $$
        commutes for every morphism $k$ in $I$. Thus the family $(f_i)_{i\in \mathrm{Obj}(I)}$ is a natural transformation $f_*F\Rightarrow G$. It is clear that $F$ respects identities. To see that $F$ respects composition, fix morphisms $p\xrightarrow{k} q\xrightarrow{l}r$ in $I$. We show that the diagram 
        $$
        \begin{tikzcd}
F(k_s) \arrow[r, "F(k)"]                                                            & F(k_t) \arrow[d, "F((l^+k^-)^+)"]  \\
F((lk)_s) \arrow[d, "F((l_-k_+)_+)"'] \arrow[u, "F((l_-k_+)_-)"] \arrow[r, "F(lk)"] & F((lk)_t)                          \\
F(l_s) \arrow[r, "F(l)"']                                                           & F(l_t) \arrow[u, "F((l^+k^-)^-)"']
\end{tikzcd}
        $$
        commutes. Consider the following diagram:
        \begin{equation}\label{cd:frogeyes}
        \begin{tikzcd}
                           & x_{(lk)_s} \arrow[ld] \arrow[ddd, "F(lk)" description] \arrow[rd] \arrow[rrr, "f_{(lk)_s}"]                       &                           &                            & y_{(lk)_s} \arrow[ld] \arrow[rd] \arrow[ddd, "G(lk)" description] &                           \\
x_{l_s} \arrow[d, "F(l)"'] &                                                                                                                   & x_{k_s} \arrow[d, "F(k)"] & y_{l_s} \arrow[d, "G(l)"'] &                                                                   & y_{k_s} \arrow[d, "G(k)"] \\
x_{l_t} \arrow[rd]         & {} \arrow[d]                                                                                                      & x_{k_t} \arrow[ld]        & y_{l_t} \arrow[rd]         &                                                                   & y_{k_t} \arrow[ld]        \\
                           & x_{(lk)_t} \arrow[d, "F((lk)_t)"'] \arrow[rrr, "f_{(lk)_t}"'] \arrow[rrrd, "\lrcorner", phantom, very near start] &                           &                            & y_{(lk)_t} \arrow[d]                                              &                           \\
                           & x \arrow[rrr, "f"']                                                                                               &                           &                            & y                                                                 &                          
\end{tikzcd}
        \end{equation}
        Composing the three paths from $x_{(lk)_s}$ to $x$ yields the same morphism, since the morphisms to $x_{(lk)_t}$ are all morphisms between suitable objects in the slice category $C/x$. Since the lower rectangle in \eqref{cd:frogeyes} is a pullback and all compositions of paths $x_{(lk)_s}\to y_{(lk)_t}$ in diagram (\ref{cd:frogeyes}) are equal, it follows from the the naturality of the family $(f_i)_{i\in\mathrm{Obj}(I)}$ that the diagram \eqref{cd:frogeyes} commutes. Hence $f^*(G)\coloneqq F$ is a functor of variance $(A,B)$.

        \item We construct a functor $G\colon I\to C/y$. For each object $i$ in $I$, we attain a morphism $f F(i)\colon x_i\to x\to y$, which we may factor as
        $$
        \begin{tikzcd}
x_{i} \arrow[d, "F(i)"'] \arrow[r, "f_i"] & y_i \arrow[d, "G(i)"] \\
x \arrow[r, "f"']                       & y                    
\end{tikzcd}
        $$
        where $G(i)f_i$ an $(E,M)$-image factorization of $f F(i)$. For a morphism $k\colon i\to j$ in $I$, we define $G(k)\colon y_{k_s}\to y_{k_t}$ as the lift induced by the orthogonality $f_{k_s§}\perp G(k_t)$. Consider the following commutative diagram:
        $$
        \begin{tikzcd}
x_{k_s} \arrow[d, "F(k)"'] \arrow[r, "f_{k_s}", two heads]   & y_{k_s} \arrow[d, "G(k)" description, hook] \arrow[dd, "G(k_s)", hook, bend left, shift left] \\
x_{k_t} \arrow[d, "F(k_t)"'] \arrow[r, "f_{k_t}", two heads] & y_{k_t} \arrow[d, "G(k_t)" description, hook]                                                 \\
x \arrow[r, "f"']                                            & y                                                                                            
\end{tikzcd}
        $$
        By the left-cancelability of $M$, we have that $G(k)\in M$. Thus, we have $G\colon I\to C/y$ defined on objects and morphisms. By the uniqueness of the lifts, $G$ respects identities. We show that $G$ respects composition.
        
        Consider then a composable pair $p\xrightarrow{k}q\xrightarrow{l}r$ in $I$. Consider again the diagram (\ref{cd:frogeyes}). We are in the same situation, but without the pullbacks. The morphism $G(lk)$ is defined as the unique lift making the relevant diagrams commute. A diagram chase, using the naturality of the family $(f_i)_{i\in\mathrm{Obj}(I)}$ and the fact that $G$ is well-defined on morphisms, shows that the required composition law for $G$ holds.\qedhere
    \end{enumerate} 
\end{proof}
\section{Coverage, Compactness, and Noetherianness}
\label{sec:DirectedCoverageCompactnessNoetherianness}

The classical notions of Noetherianness in algebra (via the ascending chain condition) and compactness in topology (via finite subcovers) are fundamental finiteness properties. While both deal with the idea that certain a priori infinite processes turn out to be finite, their standard formulations differ significantly. This section aims to introduce a unified categorical framework, termed \textit{coverages}, to study such phenomena. We extend Grothendieck's notion of a coverage to talk about compactness.

Our approach generalizes the sequential nature of chain conditions by considering processes indexed by more general small categories. A key aspect of our generalization is the explicit assignment of certain objects within these diagrams as \textit{designated small objects}. This allows us to define compactness as a stabilization property that occurs at one of these designated early stages.

Our intuition for compactness arises out of two main examples; we typically seek a finite subcover from an open cover, or for Noetherianness, a finite index at which stabilization occurs. When moving from the poset of natural numbers to more complex directed posets, such as the powerset of an arbitrary set, defining what constitutes a "small element" intrinsically can be too restrictive. Instead of relying on an inherent structural notion of finiteness within the diagram, we opt for a more synthetic approach: we explicitly choose a subset of objects within the diagram and designate them as the small ones for the purpose of our compactness definition.

Crucially, this chosen subset of designated small elements must itself be directed. This ensures that any two "finite stages" have a common "finite stage" successor, a property that proves useful when relating the compactness of different objects, for instance, in proving closure properties. The specific nature or intrinsic "finiteness" of these designated elements is less important than their role as markers for stabilization and their collective directedness.

\begin{definition}[Diagram type]\label{def:diagramType}
    Let $I$ be a small category with a variance $(\mathrm{Cov},\mathrm{Contr})$. Let $A$ be a subset of objects of $I$. We call the tuple $(I,A,\mathrm{Cov},\mathrm{Contr})$ a \textit{diagram type}, if for all $x,y\in A$ there are morphisms $x\to z\leftarrow y$ in $I$ with $z\in A$. We call the set $A$ the set of \textit{designated small objects} of the diagram type $(I,A,\mathrm{Cov},\mathrm{Contr})$. 
\end{definition}

\begin{remark}
The designated small objects $A$ serve as the "finite stages" at which stabilization can occur. The directedness condition for $A$ ensures that any two finite stages have a common finite stage successor, which is used in the proof of closure properties. This generalizes the role of natural numbers in classical chain conditions.
\end{remark}

\begin{definition}[Coverage]
\label{def:DirectedCoverage}
    Let $C$ be a category with pullbacks.
    \begin{enumerate}
        \item A family $\tau = (\tau_c)_{c \in \mathrm{Ob}(C)}$ is called a global selection on $C$, if for every $c\in\mathrm{Obj}(C)$ the class $\tau_c$ consists of tuples
        $$
        (F\colon I \to C/c, A, \mathrm{Cov},\mathrm{Contr}),
        $$
        where the $F$ is a functor of variance $(\mathrm{Cov},\mathrm{Contr})$, $C/c$ is the slice category over $c$ and the tuple $(I,A,\mathrm{Cov},\mathrm{Contr})$ is a diagram type.
        
        We will often implicitly refer to a tuple $(F,A,\mathrm{Cov},\mathrm{Contr})$ of $\tau_c$ by just $F$ and call $F\colon I\to C/c$ a \textit{covering} of $c$. Morphisms in $\mathrm{Cov}$ and $\mathrm{Contr}$ are referred to as covariant and contravariant morphisms of the covering, respectively.
        \item A global selection $\tau$ on $C$ is said to be a \textit{coverage} if for any morphism $f\colon x \to y$ in $C$ and any $(F\colon I \to C/y, A,\mathrm{Cov},\mathrm{Contr})$ covering of $y$, the pullback-induced tuple $(f^*(F)\colon I\to C/x, A,\mathrm{Cov},\mathrm{Contr})$ is a covering of $x$.\footnote{The pullback induced $f^*(F)$ defines a functor of variance by Theorem \ref{lem:pullbackinduced}}
    \end{enumerate}
\end{definition}

\begin{definition}[$\tau$-Compactness]
\label{def:TauCompactnessDirected}
    Let $C$ be a category with pullbacks and $\tau$ a global selection on $C$. An object $c$ in $C$ is \textit{$\tau$-compact} if for every covering $(F\colon I \to C/c, A,\mathrm{Cov},\mathrm{Contr})$ of $c$, the functor $F$ stabilizes at some designated small object; that is, there is a designated small object $i_0\in A$, where $F(k)$ is an isomorphism for any composable pair $i_0\to i\xrightarrow{k} j$ in $I$.
\end{definition}

\begin{definition}[Subordination and Compatibility]
\label{def:SubordinationCompatibilityDirected}
    Let $C$ be a category with a global selection $\tau$. Let $M$ and $E$ be systems of morphisms in $C$.
    \begin{enumerate}
        \item We say that a functor $F\colon I\to C/c$ of variance $(\mathrm{Cov},\mathrm{Contr})$, where $c$ is an object in $C$, is \textit{subordinated to} $M$ if $F(i)\in M$ for each object $i\in I$. Furthermore, we say the global selection $\tau$ is \textit{subordinated to $M$} if for every covering $F$ in $\tau$, the functor $F$ is subordinated to $M$. 
        \item A morphism $f\colon c \to d$ in $C$ is \textit{$(E,M)$-image compatible} with $\tau$ if for every covering $(F\colon I \to C/c, A,\mathrm{Cov},\mathrm{Contr})\in \tau_c$, there exists a covering $(G\colon I \to C/y, A,\mathrm{Cov},\mathrm{Contr})\in \tau_y$ subordinated to $M$ and there is a natural transformation $\eta: f_*F \Rightarrow G$ having components $\eta_x\colon c_x\to d_x$ in $E$. In other words, we have a diagram that commutes for every morphism $k$ in $I$:
        $$
\begin{tikzcd}
c_{k_s} \arrow[d, "F(k)"'] \arrow[r, "\eta_{k_s}\in E"]   & d_{k_s} \arrow[d, "G(k)"]        \\
c_{k_t} \arrow[d, "F(k_t)"'] \arrow[r, "\eta_{k_t}\in E"] & d_{k_t} \arrow[d, "G(k_t)\in M"] \\
c \arrow[r, "f"']                                         & d                               
\end{tikzcd}
        $$
        \end{enumerate}
\end{definition}

\begin{example}\label{ex:DirectedCoverages}
Let $C$ be a category with pullbacks and a stable system $E$. Consider the following examples:
    \begin{enumerate}
        \item \textbf{Noetherian Objects:}
        We define a coverage $\tau_M$ by setting $\tau_{M,c}$ to consist of covariant functors $F\colon \N\to C/c$ subordinated to $M$, where we consider $\N$ to be the poset of natural numbers with all natural numbers designated as small objects. The pullback stability of $M$ shows that $\tau_M$ is a coverage by Lemma \ref{lem:pullbackinduced}. In concrete terms, an object $c$ in $C$ is $\tau_M$-compact, if and only if for any commutative sequence
        $$
        \begin{tikzcd}[column sep = 2 cm, row sep = 1cm]
x_0 \arrow[rd, "F_0" description, hook, shift right] \arrow[r, "F_0^1"] & x_1 \arrow[r, "F_1^2"] \arrow[d, "F_1" description, hook] & x_2 \arrow[ld, "F_2" description, hook, shift left] \arrow[r] & ... \\
                                                                        & x                                                         &                                                               &    
\end{tikzcd}        $$
        of triangles, where $F_n\in M,$ for $n\in\N$, and there is $i_0\in\N$ so that $F_{i_0}^i$ is an isomorphism for all $i\ge i_0$. When $M$ is the class of monomorphisms of $C$, this recovers the usual notion of a Noetherian object, an object where the partially ordered class of subobjects satisfies the ascending chain condition. We call $\tau_M$-compact objects $M$-\textit{Noetherian}, and when $M$ is the class of monomorphisms, we call $\tau_M$-compact objects \textit{Noetherian}.
        
        \item \textbf{Artinian Objects:} We perturb the previous definition of $M$-Noetherianness and define $\tau_c$ to contain all contravariant functors $F\colon \N\to C/c$ with all $n\in \N$ designated as small objects. Now $\tau$-compactness of $x$ is equivalent to the stabilization of any commutative sequence
        $$
        \begin{tikzcd}
... \arrow[r] & x_2 \arrow[r, "F_1^2"] \arrow[rd, "F_2" description, hook, shift right] & x_1 \arrow[r, "F_0^1"] \arrow[d, "F_1" description, hook] & x_0 \arrow[ld, "F_0", hook, shift left] \\
              &                                                                         & x                                                         &                                        
\end{tikzcd}
        $$
        in $C$ at some $i_0\in\N$. Hence for any $i\ge i_0$ the morphism $F(i_0\to i)\colon F_i\to F_{i_0}$ is an isomorphism. We call a $\tau$-compact object $c$ $M$-Artinian, and if $M$ is the class of monomorphisms, then $c$ is called just an Artinian object.

        \item \textbf{J,M-Compactness} We may generalize further the previous two examples. Let $M$ be a stable system of $C$, and $J$ be a class of diagram types. We define a coverage $\tau_{J,M}$ by setting $\tau_{J,M,c}$ to consist of all tuples $(F\colon I\to C/c, A,\mathrm{Cov},\mathrm{Contr})$ where $F$ is a functor of variance $(\mathrm{Cov},\mathrm{Contr})$ subordinated to $M$, $(I,A,\mathrm{Cov},\mathrm{Contr})\in J$ and $c$ is an object in $C$.

        If $M$ is a stable system of $C$ and $J = \{(I,\N,\mathrm{Cov},\mathrm{Contr})\}$, where $I$ is the poset of natural numbers with the $(\mathrm{Cov},\mathrm{Contr})$ being the covariant variance and all natural numbers considered small objects, then $\tau_{J,M}$-compactness is the same as $M$-Noetherianness. If instead the variance $(\mathrm{Cov},\mathrm{Contr})$ is the contravariant variance, then we attain the $M$-Artianness as $\tau_{J,M}$-compactness.
        
        \item \textbf{Topological compactness:}
        Let $C = \mathbf{Top}$ be the category of topological spaces. Let $\kappa$ be an infinite cardinal. The coverage $\tau_\kappa$ consists of covariant functors $F\colon\mathcal{P}(I)\to \mathbf{Top}/Y$, where $\mathcal{P}(I)$ is the poset of subsets of a set $I$ with the designated small objects being subsets of $I$ with cardinality less than $\kappa$ and $F$ is freely induced from an open cover $(U_i)_{i\in I}$ of $Y$ by setting $F(A) = \bigcup_{i\in A} U_i\hookrightarrow Y$. Since a $\tau_\kappa$-covering of a space $Y$ exactly corresponds to an open cover of $Y$, we have that $\tau_\kappa$-compactness of a topological space $Y$ means that any open cover of $Y$ has a subcover of cardinality less than $\kappa$. When $\kappa = \omega, \omega+1$, where $\omega$ is the first infinite cardinal and $\omega + 1$ the second, then $\tau_\kappa$-compactness of a topological space is equivalent to compactness and Lindelöfness, respectively. Restricting the size of sets $I$ to countably infinite, we attain countable compactness.

        Consider the coverage $\tau$ on $\textbf{Top}$ of those contravariant functors $\mathcal{P}(I)\to \textbf{Top}/X$ that arise from an $I$-indexed family, whose intersection is empty, of closed subsets of $X$. The corresponding $\tau$-compactness is equivalent to the usual notion of compactness of a space.
    \end{enumerate}
\end{example}
\begin{lemma}\label{lem:weak-images-compatibility}
    Let $C$ be a category with pullbacks. Assume that $(E,M)$ is an orthogonal factorization system on $C$. Let $J$ be a class of diagram types. Then each morphism in $C$ is $(E,M)$-image compatible with the coverage $\tau_{J,M}$.
\end{lemma}
\begin{proof}
     By Lemma \ref{lem:pullbackinduced}(1) $\tau_{J,M}$ is a coverage and Lemma \ref{lem:pullbackinduced}(2) shows the compatibility.
\end{proof}

\begin{definition}
    Let $C$ be a category with classes of morphisms $E,M$, and a class of objects $A$. We say that $A$ is \textit{closed under}
    \begin{enumerate}
        \item $M$-\textit{subobjects}, if $b\in A$ and $a\to b\in M$ implies that $a\in A$.
        \item $E$-\textit{quotients}, if $a\in A$ and $a\to b\in E$ implies that $b\in A$.
        \item $E$-\textit{weak extensions}, if the existence of a pullback square
        $$
        \begin{tikzcd}
a \arrow[d, "\bar{\phi}"'] \arrow[r, "\bar{f}"] \arrow[rd, "\lrcorner", phantom, very near start] & d \arrow[d, "\phi"] \\
b \arrow[r, "f"']                                                                                 & c                  
\end{tikzcd}
        $$
        where $a,c\in A$ and $f\in E$, implies $b\in A$. We say that $A$ is \textit{closed under weak extensions}, if $A$ is closed under $B$-weak extensions, where $B$ is the class of all morphisms of $C$.
    \end{enumerate}
\end{definition}

\begin{theorem}[Closure: Subobjects, Quotients and Extensions] \label{thm:Closure}
Let $C$ be a category with pullbacks. Let $M$ and $E$ be systems of $C$ with $M$ stable and left-cancelable. Assume that $\tau$ is a coverage of $C$ subordinated to $M$. 
\begin{enumerate}
    \item(\textbf{Subobjects}) Let $J$ be a class of diagram types. If $a\to b\in M$ and $b$ is $\tau_{J,M}$-compact, then so is $a$. In particular, $M$-Noetherian and $M$-Artinian objects are closed under $M$-subobjects.
    \item (\textbf{Quotients}) Let $f\colon a\to b$ be a morphism in $C$ with $a$ $\tau$-compact. If $f$ is stably $M$-extremal, then $b$ is $\tau$-compact. Furthermore, if $E$ is right-cancelable, $E\cap M$ consists of isomorphisms and $f$ is stably in $E$, then $b$ is $\tau$-compact. 
    \item (\textbf{Weak extensions}) Assume that $E$ satisfies the $M$-protomodularity condition. Consider the pullback square
    $$
    \begin{tikzcd}
a \arrow[d, "\bar{\phi}"'] \arrow[r, "\bar{f}"] \arrow[rd, "\lrcorner", phantom, very near start] & d \arrow[d, "\phi"] \\
b \arrow[r, "f"']                                                                                 & c              
\end{tikzcd}
$$
in $C$, where $a$ and $c$ are $\tau$-compact. If $f$ is $(E,M)$-image compatible with $\tau$, then $b$ is $\tau$-compact.
\end{enumerate}
\end{theorem}\newpage
\begin{proof}\hfill
\begin{enumerate}
    \item Let $F\colon I\to C/a\in\tau_{J,M,a}$. We may post-compose $F$ with the induced functor $C/a\to C/b$ to attain a covering of $b$, and since $b$ is $\tau_{J,M}$, it follows that $F$ stabilizes. 
    
    \item We show that $y$ is $\tau$-compact. Let $G\colon I\to C/b\in \tau_b$. By pullback closure of $\tau$, we have the pullback induced $F\colon I\to C/a\in \tau_a$. Since $a$ is $\tau$-compact, $F$ stabilizes at some designated small object $i_0$. Let $k$ be a morphism in $I$ so that $F(k)$ is an isomorphism. Consider the induced commutative diagram:
    $$
    \begin{tikzcd}[column sep = 1cm]
a_{k_s} \arrow[d] \arrow[r] \arrow[d, "\cong"'] \arrow[r, "f_{k_s}"] \arrow[rd, "\lrcorner", phantom, very near start] & b_{k_s}\arrow[d, "G(k)"]   \\
a_{k_t} \arrow[r, "f_{k_t}"] \arrow[d, hook] \arrow[rd, "\lrcorner", phantom, very near start]                                                               & b_{k_t} \arrow[d, hook, "G_{k_t}"] \\
a \arrow[r, "f"']                                                                                                  & b             
\end{tikzcd}
    $$
    Since $M$ is left-cancelable, $G(k)\colon b_{k_s}\to b_{k_t}$ is in $M$. If $f$ is stably $M$-extremal, then $f_{k_t}$ is $M$-extremal and thus $G(k)$ is an isomorphism. Assume then that $E$ is right-cancelable, $E\cap M$ consists of isomorphisms, and that $f$ is stably in $E$. By stability then, we have $f_{k_s},f_{k_t}\in E$. Since $E$ is a right-cancelable system, it follows that $G(k)\in E\cap M$ and thus $G(k)$ is an isomorphism. Therefore, $b$ is $\tau$-compact.

    \item Assume that $f\colon b\to c$ is $(E,M)$-image compatible with $\tau$. We show that $b$ is $\tau$-compact. Fix a covering $G\colon I\to C/b$ of $b$. We need to show that $G$ stabilizes. Denote by $F\colon I\to C/a$ the pullback induced covering of $a$. Since $f\colon b\to c$ is $(E,M)$-image compatible with $\tau$, there is a covering $H\colon I\to C/c$ of $c$ with morphisms $f_i\colon b_i\to c_i\in E$ for objects $i$ in $I$ making the diagram
    $$
    \begin{tikzcd}[column sep = 1.5cm, row sep = 0.8cm]
a_{k_s} \arrow[d, "F(k)" description, hook] \arrow[r, "\bar{\phi}_{k_s}"] \arrow[rd, "\lrcorner", phantom, very near start] & b_{k_s} \arrow[r, "f_{k_s}", two heads] \arrow[d, "G(k)" description, hook] & c_{k_s} \arrow[d, "H(k)" description, hook] \\
a_{k_t} \arrow[r, "\bar{\phi}_{k_t}" description] \arrow[d, hook] \arrow[rd, "\lrcorner", phantom, very near start]   & b_{k_t} \arrow[r, "f_{k_t}" description, two heads] \arrow[d, hook]   & c_{k_t} \arrow[d, hook]               \\
a \arrow[r, "\bar{\phi}" description] \arrow[rd, "\bar{f}"'] \arrow[rrd, "\lrcorner", phantom, very near start]       & b \arrow[r, "f" description]                                          & c                                     \\
                                                                                                                      & a \arrow[ru, "\phi"']                                                 & {}                                   
\end{tikzcd}
    $$
    commute for each $k\colon i\to j$ in $I$. Since $a$ and $c$ are $\tau$-compact, $F$ and $H$ stabilize at $i_0$ and $i_1$, respectively, for some designated small objects $i_0$ and $i_1$. By definition, we have a designated small object $i$ with morphisms $i_1\to i\leftarrow i_2$ in $I$. Let $i\to i'\xrightarrow{k} j$. By definition, $F(k)$ and $H(k)$ are isomorphisms. Since $E$ satisfies the $M$-protomodularity condition, it suffices to show that $\bar{\phi}_{k_t}$ is some pullback along $f_{k_t}$. Consider the equalities
    $$
    \bar{\phi}_{k_t} = G_{k_t}^{-1}(\bar{\phi}) = G_{k_t}^{-1}(f^{-1}(\phi)) = f_{k_t}^{-1}(H_{k_t}^{-1}(\phi))
    $$
    of pullbacks up to an isomorphism.\qedhere
\end{enumerate}
\end{proof}

\begin{theorem}[Noetherian Pullback-Theorem] \label{thm:Hopfian}
Let $C$ be a category with pullbacks and a stable system $M$. Let $f\colon x\to x$ be a morphism where $x$ is $M$-Noetherian and $f^n$ is stably $S$-extremal epimorphism for each $n\in\N$, where $S$ is the class of sections of $C$.\footnote{We define $f^0 = id_x$ and $f^{n+1} = f\circ f^n$ for $f\colon x\to x$ and $n\in\N$.} Assume that $f$ has a fixed point $\pi_0\colon I_0\to x\in M$; more specifically $f\pi_0 = \pi_0$. Then the pullback of $f$ along $\pi_0$ is an isomorphism.
\end{theorem}
\begin{proof}
    We define the morphisms $\pi_n\colon I_n\to x$ and $f_n^0\colon I_n\to I_0$ via the following pullback
    $$
\begin{tikzcd}
I_n \arrow[d, "\pi_n"'] \arrow[r, "f_n^0"] \arrow[dr, phantom, "\lrcorner", very near start] & I_0 \arrow[d, "\pi_0"''] \\
x \arrow[r, "f^n"'] & x                                        
\end{tikzcd}
    $$
    of $f^n$ along $\pi_0$ for $n\in\N$. For $k,n\in\N,$ we define $\phi_n^{n+k}\colon I_n\rightarrow I_{n+k}$ and $f_{n+k}^n\colon I_{n+k}\rightarrow I_n$ to be the unique morphisms making the diagrams

$$
    \begin{tikzcd}
I_n \arrow[rd, "\phi_n^{n+k}" description] \arrow[rrrd, "f_n^0", bend left] \arrow[rdd, "\pi_n"', bend right] &                                                                    &  &                        &  & I_{n+k} \arrow[rd, "f_{n+k}^n" description] \arrow[rdd, "f^{k}\pi_{n+k}"', bend right] \arrow[rrrd, "f_{n+k}^0", bend left] &                                                                                                           &  &                        \\
                                                                                                              & I_{n+k} \arrow[drr, phantom, very near start, "\lrcorner"] \arrow[d, "\pi_{n+k}" description] \arrow[rr, "f_{n+k}^0"] &  & I_0 \arrow[d, "\pi_0"] &  &                                                                                                                             & I_n \arrow[rr, "f_n^0"] \arrow[rrd, "\lrcorner", phantom, very near start] \arrow[d, "\pi_n" description] &  & I_0 \arrow[d, "\pi_0"] \\
                                                                                                              & x \arrow[rr, "f^{n+k}"']                                           &  & x                      &  &                                                                                                                             & x \arrow[rr, "f^n"']                                                                                      &  & x                     
\end{tikzcd}
    $$
    commute. These morphisms are well-defined since $\pi_0$ is a fixed point of $f$. Furthermore, $f_n^0$ is a retraction with a spitting $\phi_0^n$ for each $n\in \N$. By stability, $\pi_n\in M$ for all $n\in\N$. Consider the functor $\Phi\colon \N\to C/x$, where $\Phi(n\to n+1)$ is defined by the commutative triangle
    $$
    \begin{tikzcd}
I_n \arrow[rd, "\pi_n"', hook] \arrow[rr, "\phi_n^{n+1}"] &   & I_{n+1} \arrow[ld, "\pi_{n+1}", hook'] \\
                                                          & x &                                       
\end{tikzcd}
    $$
    for each $n\in\N$. Since $x$ is $M$-Noetherian, it follows that $\Phi$ stabilizes. Thus there is $N\in\N$, where $\phi_N^{N+1}$ is an isomorphism. 

    Fix $m,n,k\in\N$ with $m\leq n$. Notice that the naturality square
    $$
    \begin{tikzcd}
I_{m+k} \arrow[r, "f_{m+k}^m"] \arrow[d, "\phi_{m+k}^{n+k}"'] & I_m \arrow[d, "\phi_m^n"] \\
I_{n+k} \arrow[r, "f_{n+k}^n"']                              & I_n                                  
\end{tikzcd}
    $$
    commutes by Lemma \ref{lem:pullbackinduced}(1). Setting $m= 0, n= 1, k = N$, we attain the equation
    $$
    f_{N+1}^1\circ \phi_{N}^{N+1} = \phi_0^1\circ f_{N}^0.
    $$
    Since $\phi_N^{N+1}$is an isomorphism, $\phi_0^1$ a section and $f_{N+1}^1$ is $S$-extremal, it follows that $\phi_0^1$ is an isomorphism. 
\end{proof}
\begin{definition}
    Let $C$ be a category with an orthogonal factorization system $(E,M)$. Let $K$ be a class of morphisms of $C$. We say that $K$ is closed under $(E,M)$-images along a morphism $f\colon a\to b$, if for any commutative square
    $$
    \begin{tikzcd}
        a' \arrow[d, "k"'] \arrow[r, "\bar{f}", two heads] & b' \arrow[d, "\bar{k}", hook] \\
        a \arrow[r, "f"']                                  & b                            
        \end{tikzcd}
    $$
    where $k\in K$ and $\bar{k}\in M$ and $\bar{f}\in E$, it holds that $\bar{k}\in K$. 
\end{definition}

\begin{lemma}\label{lem:StableImageClosure}
    Let $C$ be a category with pullbacks and an orthogonal factorization system $(E,M)$, where $M$ consists of monomorphisms. Let $\tau$ be a coverage subordinated to $M$ with each morphism in $E$ being $(E,M)$-image compatible with $\tau$. Let $K\subset M$ be a class of morphisms containing all isomorphisms. Consider the subcategory $\bar{K}$ generated by $K$. Let $\tau'$ be the restriction of $\tau$ to those coverings $F$ of $\tau$, where $F$ is subordinated to $\bar{K}$. Then the following holds:
    \begin{enumerate}
        \item If $K$ is stable, then $\bar{K}$ is a stable left-cancelable system of $C$.
        \item If $K$ is stable, then $\tau'$ is a coverage subordinated to $\bar{K}$.
        \item If $K$ is closed under $(E,M)$-images along morphisms in $E$, then so is $\bar{K}$.
        \item If $K$ is closed under $(E,M)$-images along morphisms in $E$ and every morphism in $E$ is $(E,M)$-image compatible with $\tau$, then every morphism in $E$ is $(E,\bar{K})$-image compatible with the global selection $\tau'$.
    \end{enumerate}
\end{lemma}

\begin{proof}
    \hfill
    \begin{enumerate}
        \item Assume that $K$ is a stable class of morphisms. The class $K$ contains all isomorphism of $C$ and it implies that $\bar{K}$ is a system of $C$ contained in $M$. Since each morphism in $\bar{K}$ is a composition of a finite path of morphisms in $K$, it follows inductively that $\bar{K}$ is stable. 

        It suffices to show that $K$ is left-cancelable in the following strong sense: Let $h\colon a\xrightarrow{f}b\xrightarrow{g}c$ be morphisms in $C$, where $h\in K$ and $g$ is a monic. We show that $f\in K$. Consider the pullback square
        $$
        \begin{tikzcd}
            a \arrow[r, "f"] \arrow[d, "id"'] \arrow[rd, "\lrcorner", phantom, very near start] & b \arrow[d, "g", hook] \\
            a \arrow[r, "h\in K"', hook]                                                        & c               
            \end{tikzcd}
        $$
        Thus $f\in K$ by the stability of $K$. Since $\bar{K}$ is a stable class of monics, it follows that it is left-cancelable.

        \item Assume that $K$ is stable. By previous, $\bar{K}$ is a stable system of $C$. Since $\tau$ is a coverage, it is pullback stable and by the stability of $\bar{K}$ we have that $\tau'$ is pullback stable as well and thus $\tau'$ is a coverage.
        
        \item Assume that $K$ is closed under $(E,M)$-images along morphisms in $E$. Let $f\colon a\to b\in E$ and let $k\colon c\to a$ be a morphism in $\bar{K}$. Since $k\in\bar{K}$, we have that $k$ is the composite of some path $c= a_n\xrightarrow{k_n}\cdots \xrightarrow{k_1} a_0 = a$. We define as the composite $k^i = k_1\cdots k_i$. Consider the factorization
        $$
        \begin{tikzcd}
            a_i \arrow[d, "k^i"', hook] \arrow[r, "\bar{f}_i", two heads] & b_i \arrow[d, "\bar{k}^i", hook] \\
            a \arrow[r, "f"', two heads]                                  & b                               
            \end{tikzcd}
        $$
        where $\bar{f}^i\in E$ and $\bar{k}^i\in M$. We show by induction that $k^i\in\bar{K}$ for each $i\leq n$. By assumption $\bar{k}^1\in K\subset \bar{K}$. Assume that $\bar{k}^i\in \bar{K}$ and $i<n$. We show that $\bar{k}^{i+1}\in\bar{K}$. Consider the commuting diagram
        $$
        \begin{tikzcd}
            a_{i+1} \arrow[d, "k_{i+1}"', hook] \arrow[r, "\bar{f}_{i+1}", two heads] & b_{i+1} \arrow[d, "\bar{k}_{i+1}" description] \arrow[dd, "\bar{k}^{i+1}", hook, bend left] \\
            a_i \arrow[d, "k^i"', hook] \arrow[r, "\bar{f}_i", two heads]             & b_i \arrow[d, "\bar{k}^i" description, hook]                                                \\
            a \arrow[r, "f"', two heads]                                              & b                                                                                          
            \end{tikzcd}
        $$
        where $\bar{k}_{i+1}$ is the morphism induced by the orthogonality $\bar{f}_{i+1}\perp \bar{k}^i$. By the left-cancelability of $M$, we have that $\bar{k}_{i+1}\in M$ and since $\bar{f}_i\in E$, it follows by assumption that $\bar{k}_{i+1}\in K$. Thus $\bar{k}^{i+1}\in\bar{K}$. 
        
        \item Assume that $K$ is closed under $(E,M)$-images along morphisms in $E$ and assume that each morphism in $E$ is $(E,M)$-image compatible with $\tau$. By the previous, $\bar{K}$ is closed under $(E,M)$-images along morphisms in $E$. Assume that $f\colon a\to b\in E$ and $F\in\tau'_a$. Since $f$ is $(E,M)$-image compatible with $\tau$, it follows that the images induce $f_*(F)\in \tau_b$. Since $f_*(F)$ is subordinated to $M$ and $f$ maps $\bar{K}$-subobjects of $a$ to $\bar{K}$-subobject of $b$, it follows that $f_*(F)$ is subordinated to $\bar{K}$. Thus $f_*(F)\in \tau'_b$ and so $f$ is $(E,\bar{K})$-image compatible with $\tau'$. \qedhere
    \end{enumerate}
\end{proof}

\begin{corollary}
    Let $C$ be a homological category, namely $C$ is pointed, regular and protomodular. Assume that the class of normal monomorphisms are closed under images along regular epimorphisms.\footnote{A Homological category is exact iff the images of normal subobjects are normal along regular epis \cite{JANELIDZE2002367}.} Denote by $K$ the subcategory of $C$ generated by the normal monomorphisms of $C$. Let $J$ be a class of diagram types. Consider a short exact sequence
    $$
    0\to a\to b\rightarrow c\to 0
    $$
    in $C$. Then $b$ is $J,K$-compact if and only if $a$ and $c$ are.
\end{corollary}
\begin{proof}
    Let $E$ be the class of regular epimorphisms and let $M$ be the class of monomorphisms in $C$. Note that normal monomorphisms are closed under pullback and contain isomorphisms. By assumption normal monomorphisms are closed under images along reuglar epis. Thus by Lemma~\ref{lem:StableImageClosure}(4), with $K$ the subcategory generated by the normal monomorphisms, we have that $K$ is a stable left-cancelable system, the coverage $\tau_{J,K}$ is subordinated to $K$, every morphism of $E$ is $(E,K)$-image compatible with $\tau_{J,K}$, and $E$ satisfies the $K$-protomodularity condition. Hence Theorem~\ref{thm:Closure} applies to the pair $(E,K)$.

    If $b$ is $J,K$-compact, then Theorem~\ref{thm:Closure}(1,2) gives that $a$ and $c$ are $J,K$-compact, since $a\to b\in K$ and $b\twoheadrightarrow c\in E$.

    Conversely, assume that $a$ and $c$ are $J,K$-compact. In the short exact sequence, $a\hookrightarrow b$ is a normal monomorphism and $b\twoheadrightarrow c$ is a regular epimorphism, and the square is a pullback:
    $$
    \begin{tikzcd}
    a \arrow[r] \arrow[d, hook] \arrow[rd, "\lrcorner", phantom, very near start] & 0 \arrow[d] \\
    b \arrow[r, two heads] & c
    \end{tikzcd}
    $$
    Since $b\to c\in E$ is $(E,K)$-image compatible and $a$ and $c$ are $J,K$-compact, Theorem~\ref{thm:Closure}(3) yields that $b$ is $J,K$-compact.
\end{proof}

\begin{definition}
    Let $C$ be a category with a class of morphisms $E$. We say that an object $c$ is $E$-Hopfian, if any morphism $c\to c\in E$ is an isomorphism. The object $c$ is called Hopfian, if any stably extremal epimorphism $c\to c$ is an isomorphism. 
\end{definition}

\begin{definition}
    Let $C$ be a category with pullbacks and let $\tau$ be a coverage on $C$. Let $E$ and $M$ be classes of morphisms of $C$. We say that the pair $(E,M)$ is \textit{$\tau$-well behaved}, if:
    \begin{enumerate}
        \item $\tau$ is subordinated to a stable and left-cancelable system $M$.
        \item $E$ satisfies the $M$-protomodularity condition and each morphism in $E$ is $(E,M)$-image compatible with $\tau$.
        \item At least one of the following conditions holds:
        \begin{itemize}
            \item All morphisms of $E$ are stably $M$-extremal epimorphisms.
            \item $E$ is a right-cancelable, stable system of $C$ where $E\cap M$ is the class of isomorphisms.
        \end{itemize}
    \end{enumerate}
\end{definition}

\begin{corollary}\label{cor:closureproducts}
    Let $C$ be a category with pullbacks and a coverage $\tau$. Assume that $(E,M)$ is a $\tau$-well behaved pair. Then $\tau$-compact objects are closed under $E$-quotients and $E$-weak extensions. Furthermore, if $C$ is pointed and each product projection $x\times y\to x$ is $(E,M)$-image compatible with $\tau$, then $\tau$-compact objects are closed under binary products.
\end{corollary}
\begin{proof}
    The closure of $\tau$-compact objects under $E$-quotients and $E$-weak extensions follows from Theorem \ref{thm:Closure}(2,3). If $C$ has a zero-object $0$ and the product projections in $C$ are $(E,M)$-image compatible with $\tau$, then the closure under binary products of $\tau$-compact objects is seen by considering the following pullback square for $\tau$-compact objects $a$ and $b$ in $C$ and applying Theorem \ref{thm:Closure}(3):
    $$
\begin{tikzcd}
a \arrow[d, "\iota_1"'] \arrow[r, "!"] \arrow[rd, very near start, phantom, "\lrcorner"] & 0 \arrow[d, "!"] \\
a\times b \arrow[r, "\pi_2"']                     & b               
\end{tikzcd}
    $$
    \end{proof}

\begin{corollary}
    Let $M$ be the class of extremal monomorphisms of $\mathbf{Top}$. Let $\tau$ be any coverage on $\mathbf{Top}$ subordinated by $M$. Let $f\colon X\to Y$ be a surjective continuous map with $X$ $\tau$-compact. Then $Y$ is $\tau$-compact. In particular, if $X$ is compact, countably compact, or Lindelöf, then so is $Y$, respectively. 
\end{corollary}
\begin{proof}
     Since $f$ is $M$-extremal and $\tau$ is a coverage subordinated to a stable system $M$, it follows from Theorem \ref{thm:Closure}(2) that $Y$ is $\tau$-compact. 
\end{proof}

\begin{definition}
    Let $C$ be a category with pullbacks. We say that an object $c$ in $C$ is mono-reflective, if for any morphism $f$ with domain $c$ it holds that $f$ having a monic pullback implies that $f$ is itself monic.
\end{definition}
Notice that in a protomodular category with pullbacks, each object is mono-reflective. Assume that $f\colon x\to y$ has a pullback that is monic. We may extend the pullback square via kernel pairs to attain a new pullback square as follows:
$$
\begin{tikzcd}
a' \arrow[d] \arrow[r, "\pi_2'", shift left] \arrow[r, "\pi_1'"', shift right] \arrow[rd, "\lrcorner", phantom, very near start] & x' \arrow[d] \arrow[rd, "\lrcorner", phantom, very near start] \arrow[r, "f'"] & y' \arrow[d] \\
a \arrow[r, "\pi_1"', shift right] \arrow[r, "\pi_2", shift left]                                                              & x \arrow[r, "f"']                                                              & y           
\end{tikzcd}
$$
Since $f'$ is monic, its kernel pair is a pair of isomorphisms. Thus, the retraction $\pi_1$ has a pullback that is an isomorphism, which implies directly that $\pi_1$ is an isomorphism by protomodularity. Thus, in a protomodular category, all objects are mono-reflective.

\begin{corollary}\label{cor:noeth-implies-hopfian}
    Let $C$ be a category with pullbacks and an initial object $0$. Let $x$ be a Noetherian and mono-reflective object. Assume that the unique morphism $0\to x$ factors $\pi_0e$ where $\pi_0$ is monic and $e$ is epic. Then $x$ is Hopfian, i.e., any stably extremal epimorphism $x\to x$ is an isomorphism.
\end{corollary}
\begin{proof}
    By Theorem \ref{thm:Hopfian}, the pullback of $f$ along $\pi_0$ is an isomorphism. Note that $\pi_0$ is a fixed point of $f$, since $e$ is an epimorphism. Since $x$ is mono-reflective, it follows that $f$ is monic and hence an isomorphism.
\end{proof}

\begin{corollary}\label{cor:noeth-regular-protomod-initial}
    Let $C$ be a regular protomodular category with an initial object $0$. Then the Noetherian objects of $C$ are Hopfian and closed under subobjects, regular quotients and weak extensions. Moreover, if $C$ is pointed, then the Noetherian objects of $C$ are closed under binary products.
\end{corollary}
\begin{proof}
    Consider the classes $E$ of regular epimorphisms and $M$ of monomorphisms. The class $E$ satisfies the $M$-protomodularity condition \cite{BorceuxBourn2004}. Thus $(E,M)$ is a stable, protomodular and orthogonal factorization system on $C$. The rest then follows by an application of Theorem \ref{thm:Closure} and Corollary \ref{cor:noeth-implies-hopfian}.
\end{proof}

\begin{corollary}\label{cor:compactnessinabeliansetting}
    Let $C$ be an abelian category with $M$ being the class of monomorphisms of $C$. Consider a class $J$ of diagram types. Then the full subcategory of $\tau_{J,M}$-compact objects forms an exact full abelian subcategory of $C$. In particular, the full subcategories of Artinian and Noetherian objects of $C$ are abelian categories.
\end{corollary}
\begin{proof}
    Since $C$ is abelian, it is pointed, regular and protomodular. Therefore, $\tau_{J,M}$-compact objects are closed under subobjects, quotients and biproducts by Corollary \ref{cor:noeth-regular-protomod-initial}. It is a straightforward verification that the full subcategory of $C$ generated by $\tau_{J,M}$-compact objects is an exact subcategory.
\end{proof}

\begin{corollary}
    Let $C$ be a category with finite limits and a zero-object $0$. Let $x$ and $y$ be objects of $C$ with $x$ being Noetherian and mono-reflective. If $x\cong x\times y$, then $y\cong 0$. In particular, if in addition $C$ is regular and protomodular, then $x^n\cong x^m$ implies that either $x\cong 0$ or $n = m$ for $n,m\in \N$.
\end{corollary}
\begin{proof}
    Consider the retraction $x\cong x\times y\xrightarrow{\pi_1}x$. Since $x$ is Noetherian and mono-reflective and $C$ is pointed, then, by Corollary \ref{cor:noeth-implies-hopfian}, $\pi_1$ is an isomorphism as it is a retraction. Thus the zero-morphism $y\to y$ is equal to the identity and so $y\cong 0$.

    Assume then that $C$ is regular and protomodular. Corollary \ref{cor:noeth-regular-protomod-initial} shows that the products $x^n,n\in\N$, are Noetherian. Assume that $x^m\times x^k\cong x^m$. Since $x^m$ is Noetherian, it follows that $x^k\cong 0$ by the previous part.
    
    The rest follows from the fact that $a\times b\cong 0$ implies that $a,b\cong 0$ for all objects $a$ and $b$ in $C$: Assume $a\times b\cong 0$ in $C$ and there are multiple morphisms $a\to a$. Then there are multiple morphisms $a\to a\times b$, which lead to a contradiction. So $id_a = 0_a\colon a\to a$ and thus $a\cong 0$. Similarly, $b\cong 0$.
\end{proof}

\begin{corollary}
    Let $R$ be a (left) Noetherian ring. Then the finitely generated $R$-modules are exactly the Noetherian $R$-modules.
\end{corollary}
\begin{proof}
If $M$ is a finitely generated $R$-module, then there is a surjection $R^n\to M$ for some $n\in \N$. Since Noetherian objects are closed under images and products by Corollary~\ref{cor:noeth-regular-protomod-initial}, $M$ is Noetherian. Conversely, if $M$ is a Noetherian module that is not finitely generated, then we may pick a sequence $(m_i)_{i\in\mathbb{N}}$ of elements of $M$ such that $m_i$ is not in the submodule $M_i$ generated by $(m_j)_{j<i}$ for $i\in\N$. As the sequence of submodules $(M_i)_{i\in\N}$ is increasing and non-stabilizing, a contradiction follows. Hence, $M$ is finitely generated.
\end{proof}

\section*{Appendix}
In this appendix, we show that the M-protomodularity condition, which is central to our main closure theorems, is not limited to the standard factorization systems found in regular categories. For an algebraic theory $T$ and a choice of a binary term $t$, we construct pairs $(E, M)$ of morphism classes within the category of $T$-algebras, where $E$ satisfies the $M$-protomodularity condition.

We introduce a class $E$ of (weak/strong) $t$-uniformity for morphisms of $T$-algebras. However, it is unlikely that $E$ is a part of a factorization system. The notion of $t$-uniformity of a morphism $f\colon X\to Y$ of $T$-algebras forces the equivalence classes of the kernel relation $\sim_f$, $x\sim_f y$ iff $f(x) = f(y)$ for $x,y\in X$, to be strongly correlated. In particular, it allows us to recover the property that a $t$-uniform morphism $f\colon M\to N$ is injective if its restriction $f|_{f^{-1}(I_N)}$ is injective, where $I_N$ is the minimal subalgebra of $N$.

\begin{definition}
    Let $T$ be an algebraic theory with a term $t(x,y)$, denoted $xy$. By $I_M$, \textit{the structural constants,} we refer to the smallest subalgebra of a $T$-algebra $M$. A \textit{normal subalgebra} of $M$ is an inverse-image of $I_N$ along some $T$-algebra morphism $M\to N$. We say that a $T$-algebra morphism $f\colon M\to N$ is
        \begin{enumerate}
        \item \textit{weakly $t$-uniform}, if for any $m_1,m_2\in M$,
        $$
        f(m_1) = f(m_2) \text{ implies that } m_i = mk_i, i = 1,2,
        $$
        for some $m\in M$ and $k_1,k_2\in f^{-1}(I_N)$.
        \item $t$-\textit{uniform}, if for any $m,m'\in M$,
        $$
        f(m) = f(m') \text{ implies that } m = m'k
        $$
        for some $k\in f^{-1}(I_N)$.

        \item \textit{strongly $t$-uniform}, if 
        $$
        mf^{-1}(I) = f^{-1}(f(m) I)
        $$
        for all $m\in M$ and normal subalgebras $I\leq N$. 

        \item \textit{$t$-cancelative}, if 
        $$
        f(m)f(m_1) = f(m)f(m_2) \text{ implies that } f(m_1) = f(m_2)
        $$
        for $m,m_1,m_2\in M$. We call $f$ weakly $t$-cancellative, if the previous condition holds for $m_1,m_2\in f^{-1}(I_N)$.
    \end{enumerate}
\end{definition}
\begin{remark}
These notions of $t$-uniformity are designed to capture different aspects of how the kernel relation $\sim_f$ of a morphism $f\colon M\to N$ of $T$-algebras interacts with the algebraic structure. Weak $t$-uniformity generalizes the Schreier retractions between monoids, while strong $t$-uniformity ensures compatibility with normal subalgebras. The main idea is that the equivalence classes of $\sim_f$ are strongly correlated, when $f$ is $t$-uniform. If $T$ is the theory of groups, then the equivalence classes $\sim_f$ consists of translations of the equivalence class of the identity element. These different notions of $t$-uniformity of $f$ forces certain kind of uniformity among the equivalence classes of $\sim_f$.
\end{remark}
In the case that the binary term $t(x,y)$ is right unital, in the sense that $t(x,e) \approx e$ is provable from $T$, where $e$ is a nullary term, we have an implication chain $3\Rightarrow 2\Rightarrow 1$ in the definition above. If $T$ is a pointed theory, then every morphism of $T$-algebras is weakly $t$-cancelative.

\begin{theorem}\label{thm:properties-of-uniform-morphisms}
    Let $T$ be an algebraic theory with a binary term $t(x,y)$. We denote by $I_M$ the minimal subalgebra of a $T$-algebra $M$; we also call $I_M$ the subalgebra of structural constants. Then the following assertions hold:
    \begin{enumerate}
        \item The class of surjective strong $t$-uniform morphisms forms a right-cancelable system.
        \item Let $f\colon G\to N$ be a $T$-algebra morphism, where for each $g_1,g_3\in G$ there exists $g_2\in G$ so that $g_1g_2 = g_3$ and if $f(g_1) k = f(g_3)$, then $k = f(g_2)$ for any $k\in N$. Then $f$ is a stably strong $t$-uniform morphism.
        \item Consider a commutative rectangle in $T$-\textbf{Alg}:
        $$
                    \begin{tikzcd}
f'^{-1}(I_{N'}) \arrow[d, "\alpha"'] \arrow[r, "k'", hook] & M' \arrow[d, "\beta" description] \arrow[r, "f'"] & N' \arrow[d, "\gamma"] \\
f^{-1}(I_N) \arrow[r, "k"',hook]                                                                         & M \arrow[r, "f"']                                 & N                     
\end{tikzcd}
        $$
        If $f',\alpha$ and $\gamma$ are surjections and $f$ a $t$-uniform surjection, then $\beta$ is a surjection. Moreover, if $\alpha$ and $\gamma$ are injections, $\beta$ $t$-cancelative and $f'$ weakly $t$-uniform, then $\beta$ is injective. We may weaken $\beta$ to be weakly $t$-cancelative in the previous statement, if $f$ is injective.
    \end{enumerate}
    \end{theorem}

    \begin{proof}\hfill
        \begin{enumerate}
            \item For closure under composition, consider $t$-uniform morphisms $M\xrightarrow{f} N\xrightarrow{g}P$. Let $m\in M$ and let $I\leq N$ be a normal subalgebra. We show that $m (gf)^{-1}(I) =(gf)^{-1}(gf(m)I)$, the converse is immediate. Now
            \begin{align*}
                m(gf)^{-1}(I)
                &= mf^{-1}(g^{-1}(I))\tag{Functoriality of the inverse-images}\\
                &=f^{-1}(f(m)g^{-1}(I))\tag{$g^{-1}(I)$ is a normal subalgebra}\\
                &=f^{-1}(g^{-1}(gf(m)I))\tag{$g$ strongly $t$-uniform}\\
                &= (gf)^{-1}(gf(m) I ).\tag{Functoriality of inverse-image}
            \end{align*}
            
            For closure under isomorphisms, let $f\colon M\to N$ be an isomorphism of $T$-algebras. Let $m\in M$ and $I\leq N$. We show that $f^{-1}(f(m)I)\subset mf^{-1}(I)$. Let $m'\in M$ and $k\in I$ so that $f(m') = f(m)k$. By surjectivity, $k = f(k')$ for some $k'\in M$. Now $f(m') = f(m)f(k') = f(mk')$ and by injectivity, $m' = mk'$ and so $m'\in mf^{-1}(I)$. Hence $f^{-1}(f(m)I)\subset mf^{-1}(I).$ Thus $f$ is strongly $t$-uniform. 
            
            For right-cancelability, fix a composable pair $M\xrightarrow{f}N\xrightarrow{g}P$ of $T$-algebra morphisms and assume $gf$ is strongly $t$-uniform and $f$ is surjective. We show that $g$ is strongly $t$-uniform. The surjectivity of $g$ is clear. Let $n\in N$ and let $I\leq P$ be a normal subalgebra. We show that $g^{-1}(g(n)I)\subset ng^{-1}(I)$. Since $f$ is surjective, $n = f(m)$ for some $m\in M$. Now
            \begin{align*}
               g^{-1}(g(n)I)
               &= f(f^{-1}(g^{-1}(gf(m) I))\tag{$f$ surjective, $n = f(m)$}\\
               &= f(mf^{-1}(g^{-1}(I)))\tag{$gf\in E$}\\
               &\subset f(f^{-1}(f(m)g^{-1}(I)))\tag{$mf^{-1}(J)\subset f^{-1}(f(m)J),J\subset N$}\\
               &=ng^{-1}(I).\tag{$f$ surjective, $n = f(m)$}
            \end{align*}
            
            \item Consider the following pullback in $T$-\textbf{Alg}:
             $$
             \begin{tikzcd}
            G\times_N M \arrow[d, "\pi_1"'] \arrow[r, "\pi_2"] \arrow[rd, "\lrcorner", phantom, very near start] & M \arrow[d, "h"] \\
            G \arrow[r, "f"']                                                                                    & N               
            \end{tikzcd}
             $$
             We show that $\pi_2$ is strongly $t$-uniform. Let $I\leq M$ be a normal subalgebra and let $(a,b)\in G\times_N M$. We show that $\pi_2^{-1}(b I) \subset (a,b)\pi_2^{-1}(I)$. Assume that $(x,y)\in \pi_2^{-1}(b I)$. Thus $y = bk$ for some $k\in I$ and by assumption, $ag = x$ for some $g\in G$, where $f(a)k = f(x)$ implies that $k = f(g)$ for any $k\in N$. Now $(x,y) = (a,b)(g, k)$ in $G\times M$. It suffices to show that $(g, k)\in G\times_N M$. Notice that 
             $$
             f(a)h(k) = h(b)h(k) = h(bk) = g(y) = f(x).
             $$
             and thus $f(g) = h(k)$ by the assumption on $g$ and so $(x,y)\in (a,b)\pi_2^{-1}(I)$. Therefore, $\pi_2$ is strongly $t$-uniform proving that $f$ is stably strong $t$-uniform.
             \item Let the following be a commutative diagram in $T$-\textbf{Alg}:
            \begin{equation}\label{cd:commutative-rectangle-uniformity}
            \begin{tikzcd}
            f'^{-1}(I_{N'}) \arrow[d, "\alpha"'] \arrow[r, hook] & M' \arrow[d, "\beta" description] \arrow[r, "f'"] & N' \arrow[d, "\gamma"] \\
            f^{-1}(I_N) \arrow[r,hook]                                                                         & M \arrow[r, "f"']                                 & N                     
            \end{tikzcd}
             \end{equation}
             First, we assume that $\alpha,\gamma$ and $f'$ are surjective and $f$ is a $t$-uniform surjection. We show the surjectivity of $\beta$. Let $y\in M$. By surjectivity of $\gamma f' = f\beta$ we have that $f(x) = f(\beta(x'))$ for some $x'\in M'$. By $t$-uniformity of $f$, we have that $x = \beta(x')l$ for some $l\in f^{-1}(I_N)$. We have $l = \alpha(l') = \beta(l')$ for some $l'\in f^{-1}(I_{N'})$ by surjectivity of $\alpha$ and thus $\beta(x'l') = \beta(x')\beta(l') = \beta(x')l = x$. Thereby, $\beta$ is surjective. 

             Second, we assume that $\alpha$ and $\gamma$ are injective, $f'$ weakly $t$-uniform and $\beta$ satisfies $t$-cancelation. Notice first that the left square of \eqref{cd:commutative-rectangle-uniformity} is a pullback, since $\gamma^{-1}(I_N) = I_{N'}$ by injectivity. We show that $\beta$ is injective. Assume that $\beta(m_1') = \beta(m_2')$ for $m_1',m_2'\in M'$. By injectivity of $\gamma$, we have $f'(m_1') = f'(m_2')$. Since $f'$ is weakly $t$-uniform, we have that $m_i' = m'k'_i$ for $i = 1,2$ and for some $m'\in M'$ and $k_1',k_2'\in f^{-1}(I_{N'})$. Thus $\beta(m')\beta(k'_1) = \beta(m')\beta(k'_2)$ and by $t$-cancelativity of $\beta$, $\beta(k_1') = \beta(k_2')$. Since the left square is a pullback, $k_1',k_2'\in f^{-1}(I_{N'})$ and so the injectivity of $\alpha$ shows that $k_1' = k_2'$. Therefore $m_1' = m'k_1' = m'k_2' = m_2'$ which proves the injectivity of $\beta$.

             Third, the previous argument also applies when $\beta$ is only assumed to be weakly $t$-cancellative and $f$ injective.\qedhere
             \end{enumerate}
    \end{proof}
\begin{corollary}\label{cor:uniform-morphisms-protomodular}
        Let $T$ be an algebraic theory with a binary term $t(x,y)$. Then in the following two cases the pair $(E,M)$ of classes of morphisms of $T$-algebras satisfy the protomodularity condition, where $E$ and $M$ are the classes of 
        \begin{enumerate}
            \item $t$-uniform surjections and injections, respectively,
            \item weakly $t$-uniform morphisms and $t$-cancelative surjections, respectively,
        \end{enumerate}
    \end{corollary}
    \begin{proof}
        The claim follows directly from Theorem \ref{thm:properties-of-uniform-morphisms}(3).
    \end{proof}

    Notice that if the theory $T$ is pointed and $t(x,e)\approx x$ is provable from $T$ for some constant term $e$, then the class of strong $t$-uniform morphisms forms a protomodular system of $\textbf{Set}^T$.

\begin{corollary}\label{cor:monic-pullback}
    Let $T$ be an algebraic theory with a binary term $t(x,y)$. Let $f\colon M\to N$ be a weakly $t$-uniform and weakly $t$-cancellative. Then $f$ is injective if and only if the restriction of $f$ to $f^{-1}(I)$ is injective.
\end{corollary}
\begin{proof}
   Consider the diagram
    $$
    \begin{tikzcd}
f^{-1}(I_N) \arrow[d, "f\mid"'] \arrow[r, hook] & M \arrow[d, "f" description] \arrow[r, "f"] & N \arrow[d, "id"] \\
I_N \arrow[r, hook]                             & N \arrow[r, "id"']                          & N                
\end{tikzcd}
    $$
    Since $id_N$ and $f\mid$ are injective, $f$ both weakly $t$-uniform and weakly $t$-cancellative, we have by Theorem \ref{thm:properties-of-uniform-morphisms}(3) that $f$ is injective.
\end{proof}

\begin{corollary}
Let $T$ be an algebraic theory with a binary term $t(x,y)$. Let $f\colon M\to M$ be a surjective weakly $t$-uniform and weakly $t$-cancelative morphism between $T$-algebras. If $M$ is a Noetherian $T$-algebra, then $f$ is an isomorphism.
\end{corollary}
\begin{proof}
    If $M$ is Noetherian, then by Theorem \ref{thm:Hopfian} $f$ has a pullback that is an isomorphism, and Corollary \ref{cor:monic-pullback} shows that $f$ is injective and thus an isomorphism.
\end{proof}

\section*{Acknowledgements}
The author gratefully acknowledges the research was funded by the Fonds de la Recherche Scientifique (Belgium) through an Aspirant fellowship.

\bibliography{references}

\begin{thebibliography}{7}
\providecommand{\natexlab}[1]{#1}
\providecommand{\url}[1]{\texttt{#1}}
\expandafter\ifx\csname urlstyle\endcsname\relax
  \providecommand{\doi}[1]{doi: #1}\else
  \providecommand{\doi}{doi: \begingroup \urlstyle{rm}\Url}\fi

\bibitem[Borceux and Bourn(2004)]{BorceuxBourn2004}
Francis Borceux and Dominique Bourn.
\newblock \emph{{M}al'cev, Protomodular, Homological and Semi-Abelian Categories}, volume 566 of \emph{Mathematics and Its Applications}.
\newblock Springer, 2004.
\newblock \doi{10.1007/978-1-4020-1962-3}.
\newblock URL \url{https://link.springer.com/book/10.1007/978-1-4020-1962-3}.

\bibitem[Bourn et~al.(2016)Bourn, Martins-Ferreira, Montoli, and Sobral]{MonoidsAndPointedSProtomodular}
Dominique Bourn, Nelson Martins-Ferreira, Andrea Montoli, and Manuela Sobral.
\newblock Monoids and pointed $s$-protomodular categories.
\newblock \emph{Homology, Homotopy and Applications}, 18:\penalty0 151--172, 01 2016.
\newblock \doi{10.4310/HHA.2016.v18.n1.a9}.

\bibitem[Forsman(2023)]{forsman2023functorsvarianceheuristicnaturality}
David Forsman.
\newblock Functors of variance: Heuristic naturality and connection to ends, 2023.
\newblock URL \url{https://arxiv.org/abs/2305.05361}.

\bibitem[Gran(2021)]{inbook}
Marino Gran.
\newblock \emph{An Introduction to Regular Categories}, pages 113--145.
\newblock 10 2021.
\newblock ISBN 978-3-030-84318-2.
\newblock \doi{10.1007/978-3-030-84319-9_4}.

\bibitem[Janelidze et~al.(2002)Janelidze, Márki, and Tholen]{JANELIDZE2002367}
George Janelidze, László Márki, and Walter Tholen.
\newblock Semi-abelian categories.
\newblock \emph{Journal of Pure and Applied Algebra}, 168\penalty0 (2):\penalty0 367--386, 2002.
\newblock ISSN 0022-4049.
\newblock \doi{https://doi.org/10.1016/S0022-4049(01)00103-7}.
\newblock URL \url{https://www.sciencedirect.com/science/article/pii/S0022404901001037}.
\newblock Category Theory 1999: selected papers, conference held in Coimbra in honour of the 90th birthday of Saunders Mac Lane.

\bibitem[Janelidze(2006)]{RelativeHomologicalCategories}
Tamar Janelidze.
\newblock Relative homological categories.
\newblock \emph{Journal of Homotopy and Related Structures}, 1, 10 2006.

\bibitem[Kelly(1991)]{Kelly1991}
G.~M. Kelly.
\newblock A note on relations relative to a factorization system.
\newblock In Aurelio Carboni, Maria~Cristina Pedicchio, and Giuseppe Rosolini, editors, \emph{Category Theory}, volume 1488 of \emph{Lecture Notes in Mathematics}, pages 249--261. Springer, Berlin, Heidelberg, 1991.
\newblock \doi{10.1007/BFb0084224}.
\newblock URL \url{https://link.springer.com/chapter/10.1007/BFb0084224}.
\newblock First online: 01 January 2006.

\end{thebibliography}

\end{document}